\documentclass[11pt]{article}
\usepackage{amsmath,amsthm,amsfonts,amssymb,amscd, amsxtra}

\newcommand{\banacha}{\mathbb X}
\newcommand{\banachb}{\mathbb Y}

\newtheorem{theorem}{Theorem}
\newtheorem{lemma}[theorem]{Lemma}

\newtheorem{corollary}[theorem]{Corollary}
\newtheorem{proposition}[theorem]{Proposition}

\newtheorem{remark}{Remark}
\begin{document}
\title{Local convergence analysis of inexact Gauss-Newton like \\   methods under  majorant condition}

\author{ O. P. Ferreira\thanks{IME/UFG, Campus II- Caixa
    Postal 131, 74001-970 - Goi\^ania, GO, Brazil (E-mail:{\tt
      orizon@mat.ufg.br}). The author was partly supported by
    CNPq Grant 473756/2009-9, CNPq Grant 302024/2008-5, PRONEX-Optimization(FAPERJ/CNPq) and FUNAPE/UFG.}
   \and  M. L. N. Gon\c calves \thanks{COPPE-Sistemas, Universidade Federal do Rio de Janeiro, 21945-970 Rio de Janeiro, RJ, BR (E-mail:{\tt
      maxlng@cos.ufrj.br}). The author was partly supported by
    CNPq Grant 473756/2009-9.} \and P. R. Oliveira \thanks{COPPE-Sistemas, Universidade Federal do Rio de Janeiro,
21945-970 Rio de Janeiro, RJ, BR (Email: {\tt poliveir@cos.ufrj.br}).
This author was partly supported by CNPq.} }
 \maketitle
\begin{abstract}
In this paper, we present a local convergence analysis of inexact Gauss-Newton like methods for solving nonlinear least squares problems. Under the hypothesis that the derivative of the function associated with the least square problem satisfies a majorant condition, we obtain that the method is well-defined and converges.  Our analysis provides a clear relationship between the majorant function and the function associated with the least square problem. It also allows us to obtain an estimate of convergence ball for inexact Gauss-Newton like methods and some important, special cases.
\end{abstract}


\maketitle
\section{Introduction}\label{sec:int}
Let $\banacha$ and $\banachb$ be real or complex Hilbert spaces. Let $\Omega\subseteq\banacha$ be an open set, and  $F:\Omega\to \banachb$ a continuously differentiable nonlinear function. Consider the following {\it nonlinear least squares} problems
\begin{equation}\label{eq:p1}
\min_{x\in \Omega } \;\|F(x)\|^2.
\end{equation}
The interest in this problem arises in data fitting, when $\banacha=\mathbb{R}^n$ and $\banachb=\mathbb{R}^m$ and $m$ is the number of observations and $n$ is the number of parameters, see for example \cite{DN1}. A solution $x_* \in \Omega$ of \eqref{eq:p1} is also called a least-squares solution of nonlinear equation $F(x)=0.$

When $F'(x)$ is injective and has closed image for all $x\in \Omega$, the Gauss-Newton's method finds stationary points of
the above problem.  Formally, the Gauss-Newton's method is described as follows: Given
an initial point $x_0 \in \Omega$, define
$$
x_{k+1}={x_k}+S_{k}, \qquad F'(x_k)^*F'(x_k)S_{k}=-F'(x_k)^*F(x_k),
\qquad k=0,1,\ldots,
$$
where $A^*$ denotes the adjoint of the operator $A$.  It is worth pointing out that if $x_*$ is solution of \eqref{eq:p1}, $F(x_*)=0$ and $F'(x_*)$ is invertible, then the theories of the Gauss-Newton's method merge into the theories of Newton's method. Early works dealing with the convergence of Newton's and Gauss-Newton's methods include \cite{MR1918655,1390141,Argyros2010,C11,chen2,MR1895083,Jean-PierreDedieu07012003,MR1651750,F08,MAX2,FS06,KAN1,Li2010,LZJ,MR2475307,Proinov20103,XW10}.

The inexact Gauss-Newton  process is described as follows: Given
an initial point $x_0 \in \Omega$, define
$$
x_{k+1}=x_{k}+S_{k},\qquad {k=0,1,...,}
$$
where $ B_{k} : \banacha \to \banachb$ is a linear operator and $S_{k}$ is any approximated solution of the linear system
$$
 {B_{k}S_{k}=-F'(x_k)^*F(x_{k})+r_{k},}
$$
for a suitable residual $r_k \in \banachb$.  In particular, the above process is  {\it inexact  Gauss-Newton method} if
$B_k=F'(x_k)^TF'(x_k),$ the process is {\it inexact modified Gauss-Newton method} if
$B_k=F'(x_0)^TF'(x_0),$ and it represents a {\it inexact  Gauss-Newton like method} if $B_k$ is an approximation of
$F'(x_k)^TF'(x_k).$

For inexact Newton methods, as shown in \cite{DE1}, if $\|r_{k}\|\leq \theta_{k}\|F(x_{k})\|$ for $k=0,1,\ldots$ and $\{\theta_{k}\}$ is a sequence of forcing terms such that $0\leq \theta_{k}<1$ then there exists $\epsilon>0$ such that the sequence $\{ x_k\}$, for any initial point $x_0 \in B(x_*, \epsilon)=\{x\in \mathbb{R}^{n}:\; \|x_{*}-x\|<\epsilon\}$, is well defined and converges linearly to  $x_{*}$  in the norm $\|y\|_*=\|F'(x_*)y\|$, where $\| \; \|$ is any norm in $ \mathbb{R}^{n}$. As pointed out by \cite{JM10} (see also \cite{B10}) the result of \cite{DE1} is difficult to apply due to a dependence of the norm  $\|\; \|_*$, which is not computable.

Formally, the inexact Gauss-Newton like methods for solving  \eqref{eq:p1}, which we will consider, are described as follows: Given an initial point $x_0 \in {\Omega}$, define
$$
x_{k+1}={x_k}+S_k,\qquad B(x_k)S_k=-F'(x_k)^*F(x_{k})+r_{k}, \qquad k=0,1,\ldots,
$$
where $B(x_k)$ is a suitable invertible approximation of the derivative $F'(x_k)^*F'(x_{k})$ and the residual tolerance  $r_k$ and the preconditioning invertible matrix $P_{k}$ (considered for the first time in \cite{B10}) for the linear system defining the step $S_k$ satisfy
$$
\|P_{k}r_{k}\|\leq \theta_{k}\|P_{k}F'(x_k)^*F(x_{k})\|,
$$
for suitable forcing number $\theta_{k}$.  Note that, if the forcing sequence vanishes, i.e., $\theta_{k}=0$ for all $k,$ the inexact Gauss-Newton methods include the class of Gauss-Newton iterative methods. Hence, the theories of inexact Gauss-Newton methods merge into the theories of Gauss-Newton methods. 

The classical local convergence analysis for the inexact Newton's methods (see \cite{DE1,B10}) requires, among other hypotheses, that $F'$ satisfies the Lipschitz condition.  In the last years, there have been papers dealing with the issue of convergence of the Newton method and inexact Newton's method, including the Gauss-Newton's method and inexact Gauss-Newton's method, by relaxing the assumption of Lipschitz continuity of the derivative (see for example: \cite{AAA,C11,C10,MR1895083,F08,MAX1,MAX2,FS06,LZJ,XW10}). One of the main conditions that relaxes the condition of the Lipschitz continuity of the derivative is the majorant condition, which we will use, and Wang's condition, introduced
in \cite{XW10} and used in \cite{AAA,C11,C10,chen2,Li2010,LZJ} to study the Gauss-Newton's and Newton's methods. In fact, it can be shown that these conditions are equivalent. But the formulation as a majorant condition is in some sense better than Wang's condition, as it provides a clear relationship between the majorant
function and the nonlinear function under consideration. Besides, the majorant condition provides a simpler proof of convergence. 

In the present paper, we are interested in the local convergence analysis, i.e., based on the information in a neighbourhood of a stationary point of \eqref{eq:p1} we determine the convergence ball of the method. Following the ideas of \cite{F08,MAX1,MAX2,FS06}, we will present a new local convergence
analysis for inexact Gauss-Newton like methods under majorant condition.  The convergence analysis
presented provides a clear relationship between the majorant
function, which relaxes the Lipschitz continuity of the derivate, and the function associated
with the nonlinear least square problem (see for example: Lemmas \ref{wdns}, \ref{pr:taylor} and \ref{passonewton}). Besides, the results presented here have the conditions and the proof of convergence in quite a simple manner.   Moreover, two unrelated previous results pertaining to inexact Gauss-Newton like methods are unified, namely, the result for analytical functions and the classical one for functions with Lipschitz derivative.

The organization of the paper is as follows.
In Section \ref{sec:int.1}, we list some notations and basic results used in our presentation.
In Section \ref{lkant} the main result is stated, and in Section \ref{sec:PMF} some properties involving the majorant function are established. In Section \ref{sec:MFNLO} we present the relationships between the majorant function and the non-linear function $F$. In Section \ref{sec:proof} the main result is proven and some applications of this result are given in Section \ref{sec:ec}. Some final remarks are offered in Section~\ref{rf}.

\subsection{Notation and auxiliary results} \label{sec:int.1}
The following notations and results are used throughout our
presentation.   Let $\banacha$ and $\banachb$ be Hilbert spaces. The open and closed ball
at $a \in \banacha$ and radius $\delta>0$ are denoted, respectively by
$$
B(a,\delta) :=\{ x\in \banacha ;\; \|x-a\|<\delta \}, \qquad B[a,\delta] :=\{ x\in \banacha ;\; \|x-a\|\leqslant \delta \}.
$$
The set $\Omega\subseteq\banacha$ is an open set and the function $F:\Omega\to \banachb$ is continuously differentiable, and $F'(x)$  has closed image in $\Omega$.

Let $A: \banacha \to \banachb$ be a continuous and injective linear operator with closed image. The Moore-Penrose inverse $A^\dagger:\banachb \to \banacha$ of $A$ is defined by
$$
A^\dagger:=(A^*A)^{-1} A^*,
$$
where  $A^*$ denotes the adjoint of the linear operator $A$.
\begin{lemma}(Banach's Lemma) \label{lem:ban1}
Let $B:\banacha \to \banacha$ be a continuous linear operator, and $\mbox{I}: \banacha \to \banacha$ the identity operator.
If  $\|B-I\|<1$,  then $B$ is invertible and $ \|B^{-1}\|\leq
1/\left(1- \|B-I\|\right). $
\end{lemma}
\begin{proof}  See the proof of Lemma 1, p. 189 of Smale \cite{S86} with $A=I$ and $c=\|B-I\|$.
\end{proof}
\begin{lemma} \label{lem:ban2}
Let $A, B: \banacha \to \banachb$ be a continuous linear operator with closed image. If $A$ is injective, $E=B-A$ and
$\|EA^\dagger \|<1$, then $B$ is injective.
\end{lemma}
\begin{proof} In fact, $B=A+E=(I+EA^\dagger)A,$ from the condition $\|EA^\dagger\|<1,$ we have of Lemma \ref{lem:ban1} that $I+EA^\dagger$ is invertible. So, $B$ is injective.
\end{proof}
The next lemma is proven in Stewart \cite{G1} ( see also, Wedin
\cite{W1} ) for  $m\times n$ matrix with $m\geq n$ and $rank(A)=rank(B)=n$, that proof holds in more general context as we will state below.
\begin{lemma} \label{lem:ban}
Let $A, B: \banacha \to \banachb$ be continuous and injective linear operators with closed images.  Assume that $E=B-A$ and $\|A^\dagger\|\|E\|<1$, then
$$
\|B^\dagger\|\leq \frac{\|A^\dagger\|}{ 1- \|A^\dagger\|\|E\|}, \qquad \|B^\dagger-A^\dagger\|\leq \frac{\sqrt{2}\|A^\dagger\|^2\|E\|}{
1- \|A^\dagger\|\|E\|}.
$$
\end{lemma}
\begin{proposition} \label{le:ess}
If $0\leq t <1$, then $ \sum
_{i=0}^{\infty}(i+2)(i+1)t^{i}=2/(1-t)^3. $
\end{proposition}
\begin{proof} Take $k=2$ in Lemma 3, pp. 161 of Blum, et al. \cite{BCSS97}.
\end{proof}
Also, the following auxiliary results of elementary convex analysis
will be needed:
\begin{proposition}
  \label{pr:conv.aux2}
Let $R>0$. If $\varphi:[0, R)\to \mathbb{R}$ is convex, then
$$
D^+ \varphi(0)={\lim}_{u\to 0+} \; \frac{\varphi(u)-\varphi(0)}{u}
={\inf}_{0<u} \;\frac{\varphi(u)-\varphi(0)}{u}. \\
$$
\end{proposition}
\begin{proof} See Theorem 4.1.1 on pp. 21 of Hiriart-Urruty and Lemar\'echal \cite{HL93}.
\end{proof}
\begin{proposition}\label{pr:conv.aux1}
Let $\epsilon>0$ and $\tau \in [0,1]$. If $\varphi:[0,\epsilon)
\rightarrow\mathbb{R}$ is convex, then $l:(0,\epsilon) \to
\mathbb{R}$ defined by
$$
l(t)=\frac{\varphi(t)-\varphi(\tau t)}{t},
$$
is increasing.
\end{proposition}
\begin{proof} See Theorem 4.1.1 and Remark 4.1.2 on pp. 21 of Hiriart-Urruty and Lemar\'echal \cite{HL93}.
\end{proof}
\section{Local analysis for inexact Gauss-Newton like methods } \label{lkant}
In this section, we will state and prove a local theorem for inexact Gauss-Newton like
methods. Assuming that the function $$\Omega \ni x  \mapsto F(x)^*F(x),$$ has a point stationary $x_*$, we will, under mild conditions, prove that the inexact Gauss-Newton like methods is well defined and that the generated sequence converges linearly to this point stationary. The statement of the theorem is as follows:

\begin{theorem}\label{th:nt}
Let $\Omega\subseteq \banacha$ be an open set,
$F:{\Omega}\to \banachb$ a continuously differentiable
function. Let $x_* \in \Omega,$ $R>0$ and
$$
c:=\|F(x_*)\|, \qquad \beta:=\left\|F'(x_*)^{\dagger}\right\|, \qquad \kappa:=\sup \left\{ t\in [0, R): B(x_*, t)\subset\Omega \right\}.
$$
Suppose that $F'(x_*)^*F(x_*)=0$,
$F '(x_*)$ is injective
and there exists a $f:[0,\; R)\to \mathbb{R}$
continuously differentiable such that
  \begin{equation}\label{Hyp:MH}
\left\|F'(x)-F'(x_*+\tau(x-x_*))\right\| \leq
f'\left(\|x-x_*\|\right)-f'\left(\tau\|x-x_*\|\right),
  \end{equation}
  for  all $\tau \in [0,1]$,  $x\in B(x_*, \kappa)$  and
\begin{itemize}
  \item[{\bf h1)}]  $f(0)=0$ and $f'(0)=-1$;
  \item[{\bf  h2)}]  $f'$ is convex and  strictly increasing;
   \item[{\bf  h3)}]  $\alpha:=\sqrt{2} \,c\, \beta ^2 D^+ f'(0)<1$.
\end{itemize}
Take $0\leq \vartheta<1$, $0\leq \omega_{2}<\omega_{1}$ such that  $\omega_{1}(\alpha+\alpha\vartheta+\vartheta)+\omega_{2}<1 $.
Let the positive constants
$$\nu  :=\sup \left\{t \in[0, R):\beta[f'(t)+1]<1\right\},$$
$$
 \rho :=\sup \bigg\{t\in(0,\nu):(1+\vartheta)\omega_{1}\beta\frac{tf'(t)-f(t)+\sqrt{2}c
 \beta[f'(t)+1]}{t[1-\beta(f'(t)+1)]}+\omega_{1}\vartheta+\omega_{2}<1\bigg\}, \; r :=\min  \left\{\kappa, \, \rho \right\}.  $$
Then, the inexact Gauss-Newton like methods for solving \eqref{eq:p1}, with initial point $x_0\in
B(x_*, r)\backslash \{x_*\}$
\begin{equation} \label{eq:DNSqn}
x_{k+1}={x_k}+S_k, \qquad  B(x_k)S_k=-F'(x_k)^*F(x_k)+r_{k}, \qquad
\; k=0,1,\ldots,
\end{equation}
for the forcing term $\theta_k$ and the following conditions for the residual $r_k$ and the invertible matrix $P_{k}$ preconditioning the linear system in \eqref{eq:DNSqn}
$$
\|P_{k}r_{k}\|\leq \theta_{k}\|P_{k}F'(x_{k})^*F(x_{k})\|,
\qquad  0\leq \theta_{k}\mbox{cond}(P_{k}F'(x_{k})^*F'(x_{k}))\leq
\vartheta,\qquad
\; k=0,1,\ldots,$$
where $B(x_k)$ is an invertible approximation of $F'(x_k)^*F'(x_k)$ satisfying the following conditions
$$
\|B(x_k)^{-1}F'(x_k)^*F'(x_k)\| \leq \omega_{1}, \qquad
\|B(x_k)^{-1}F'(x_k)^*F'(x_k)-I\| \leq \omega_{2}, \qquad
\; k=0,1,\ldots,
$$
is well defined, contained in $B(x_*,r)$, converges to $x_*$ and there holds
  \begin{multline} \label{eq:q2}
    \|x_{k+1}-x_*\| \leq
    (1+\vartheta)\omega_{1}\beta\frac{[f'(\|x_0-x_*\|)\|x_0-x_*\|-f(\|x_0-x_*\|)]}{\|x_0-x_*\|^2[1-\beta(f'(\|x_0-x_*\|)+1)]}{\|x_k-x_*\|}^{2}\\+
    \left(\frac{(1+\vartheta)\omega_1\sqrt{2}c\beta^2[f'(\|x_0-x_*\|)+1]}{\|x_0-x_*\|[1-\beta(f'(\|x_0-x_*\|)+1)]}+\omega_1\vartheta+\omega_2\right)\|x_k-x_*\|,\qquad
    k=0,1,\ldots.
\end{multline}
\end{theorem}
\begin{remark}
In particular, if taking $\vartheta=0$ (in this case $\theta_k\equiv 0$ and $r_k\equiv 0$) in Theorem~\ref{th:nt}, we obtain the convergence of Gauss-Newton's like method under majorant condition which, for $\omega_1=1$ and $\omega_2=0$, i.e., $B(x_k)=F'(x_k)^*F'(x_k)$, has been obtained by Ferreira et al. \cite{MAX2}
in Theorem~7.  Now, if taking  $c=0$ (the so-called zero-residual case) and $F'(x_*)$ is invertible, we obtain the convergence of inexact Newton-Like methods under majorant condition, which has been obtained by Ferreira, Gon\c calves \cite{MAX1} in Theorem 4. Finally, if $c=\vartheta=\omega_2=0$, $\omega_1=1$  and $F'(x_*)$ is invertible in Theorem \ref{th:nt}, we obtain the
convergence of  Newton's method under majorant condition, which has been obtained by Ferreira
\cite{F08} in Theorem~2.1.
\end{remark}
For the important case $\vartheta=0$, namely, Gauss-Newton's like method under majorant condition, the Theorem~\ref{th:nt} becomes:
\begin{corollary} \label{col:pc1}
Let $\Omega\subseteq \banacha$ be an open set,
$F:{\Omega}\to \banachb$ a continuously differentiable
function. Let $x_* \in \Omega,$ $R>0$ and
$$
c:=\|F(x_*)\|, \qquad \beta:=\left\|F'(x_*)^{\dagger}\right\|, \qquad \kappa:=\sup \left\{ t\in [0, R): B(x_*, t)\subset\Omega \right\}.
$$
Suppose that $F'(x_*)^*F(x_*)=0$,
$F '(x_*)$ is injective
and there exists a $f:[0,\; R)\to \mathbb{R}$
continuously differentiable such that
$$
\left\|F'(x)-F'(x_*+\tau(x-x_*))\right\| \leq
f'\left(\|x-x_*\|\right)-f'\left(\tau\|x-x_*\|\right),
 $$
for  all $\tau \in [0,1]$,  $x\in B(x_*, \kappa)$ and
\begin{itemize}
  \item[{\bf h1)}]  $f(0)=0$ and $f'(0)=-1$;
  \item[{\bf  h2)}]  $f'$ is convex and strictly increasing;
   \item[{\bf  h3)}]  $\alpha:=\sqrt{2} \,c\, \beta ^2 D^+ f'(0)<1$.
\end{itemize}
Take $0\leq \omega_{2}<\omega_{1}$  such that  $\omega_{1}\alpha+\omega_{2}<1 $.
Let $\nu  :=\sup \left\{t \in[0, R):\beta[f'(t)+1]<1\right\},$
$$
 \rho :=\sup \bigg\{t\in(0,\nu):\omega_{1}\beta\frac{tf'(t)-f(t)+\sqrt{2}c
 \beta[f'(t)+1]}{t[1-\beta(f'(t)+1)]}+\omega_{2}<1\bigg\}, \quad r :=\min  \left\{\kappa, \, \rho \right\}.  $$
Then, the  Gauss-Newton's like method for solving \eqref{eq:p1}, with initial point $x_0\in
B(x_*, r)\backslash \{x_*\}$
$$
x_{k+1}={x_k}+S_k, \qquad  B(x_k)S_k=-F'(x_k)^*F(x_k), \qquad
\; k=0,1,\ldots,
$$
where $B(x_k)$ is an invertible approximation of $F'(x_k)^*F'(x_k)$ satisfying
$$
\|B(x_k)^{-1}F'(x_k)^*F'(x_k)\| \leq \omega_{1}, \qquad
\|B(x_k)^{-1}F'(x_k)^*F'(x_k)-I\| \leq \omega_{2}, \qquad
\; k=0,1,\ldots,
$$
is well defined, contained in $B(x_*,r)$, converges to $x_*$ and there holds
  \begin{multline} \label{eq:q2}
    \|x_{k+1}-x_*\| \leq
    \omega_{1}\beta\frac{[f'(\|x_0-x_*\|)\|x_0-x_*\|-f(\|x_0-x_*\|)]}{\|x_0-x_*\|^2[1-\beta(f'(\|x_0-x_*\|)+1)]}{\|x_k-x_*\|}^{2}\\+
    \left(\frac{\omega_1\sqrt{2}c\beta^2[f'(\|x_0-x_*\|)+1]}{\|x_0-x_*\|[1-\beta(f'(\|x_0-x_*\|)+1)]}+\omega_2\right)\|x_k-x_*\|,\qquad
    k=0,1,\ldots.
\end{multline}
\end{corollary}
\begin{remark}
Despite the fact that the above corollary is a special case of Theorem \ref{th:nt}, the results contained therein extend the results of Chen and Li in \cite{chen2}, as the results obtained \cite{chen2} are only for the case $c=0.$
\end{remark}
\begin{remark}
Assumption \eqref{Hyp:MH} is crucial for our analysis. It should be pointed that, under appropriate regularity conditions in the nonlinear function $F$, assumption \eqref{Hyp:MH} always holds on a suitable neighbourhood of $x_*$. For instance, if  $F$ is two times continuously differentiable, then the majorant function $f:[0, \kappa)\to \mathbb{R}$, as defined by
$
f(t)=Kt^{2}/2-t,
$ where $K=\sup\{\|F''(x)\|: x\in B[x_*, \kappa]\}$ satisfies assumption \eqref{Hyp:MH}.  Estimating the constant $K$ is a very difficult problem. Therefore, the goal is to identify classes of nonlinear functions for which it is possible to obtain a majorant function.  We will give some examples of such classes in Section \ref{sec:ec}.
\end{remark}

To prove Theorem \ref{th:nt} we need some results. From
here on, we assume that all assumptions of Theorem \ref{th:nt}
hold.
\subsection{The majorant function } \label{sec:PMF}
In this section, we will prove that the constant $\kappa$ associated
with $\Omega$ and the constants  $\nu$,  $\rho$ and $r$
associated with the majorant function $f$ are positive. We will also
 prove some results related to the function $f$.

We begin by noting that $\kappa>0$, because $\Omega$ is an open
set and $x_*\in \Omega$.
\begin{proposition}  \label{pr:incr1}
The constant $  \nu $ is positive and there holds
$$
\beta[f'(t)+1]<1, \qquad t\in (0, \nu).
$$
\end{proposition}
\begin{proof}
 As $f'$ is continuous in $(0,R)$ and $f'(0)=-1,$ it is easy to conclude that
$$\lim _{t \to 0}\beta[f'(t)+1]=0. $$
Thus, there exists a $\delta>0$ such that $\beta(f'(t)+1)<1$ for
all $t\in (0, \delta)$. Hence,  $\nu>0.$

Using {\bf  h2} and definition of $\nu$ the last part of the proposition follows.

\end{proof}
\begin{proposition} \label{pr:incr101}
The following functions are increasing:

\begin{itemize}
 \item[{\bf i)}] $[0,\, R) \ni t \mapsto 1/[1-\beta(f'(t)+1)];$
 \item[{\bf ii)}]  $(0,\, R) \ni t \mapsto [tf'(t)-f(t)]/t^2;$
\item[{\bf iii)}]  $(0,\, R) \ni t \mapsto [f'(t)+1]/t;$
\end{itemize}
As a consequence, there is an increase of the following functions
$$
(0,\, R)\ni t\mapsto \frac{tf'(t)-f(t)}{t^2[1-\beta(f'(t)+1)]},
\qquad  \qquad (0,\, R)\ni t\mapsto
\frac{f'(t)+1}{t[1-\beta(f'(t)+1)]}.
$$
\end{proposition}
\begin{proof}
 The item~{\bf i} is immediate, because $f'$ is strictly increasing in $[0, R)$.

 For proving item~{\bf ii}, note that after some simple algebraic manipulations we have
  $$
  \frac{tf'(t)-f(t)}{t^2}=\int_{0}^{1}\frac{f'(t)-f'(\tau t)}{t}\; d\tau.
  $$
  So, applying
Proposition~\ref{pr:conv.aux1} with $f'=\varphi$ and $\epsilon=R$ the statement follows.

For establishing  item~{\bf iii} use  $\bf h2$, $f'(0)=-1$ and
Proposition~\ref{pr:conv.aux1} with $f'=\varphi,$ $\epsilon=R$ and
$\tau =0.$

To prove that the functions in the last part are increasing, combine item {\bf i} with {\bf ii} for the first function, and {\bf i} with {\bf iii} for the second function.
\end{proof}
\begin{proposition}\label{pr:incr102}
The constant $ \rho $ is positive and there holds
$$
(1+\vartheta)\omega_{1}\beta\frac{tf'(t)-f(t)+\sqrt{2}c
 \beta[f'(t)+1]}{t[1-\beta(f'(t)+1)]}+\omega_{1}\vartheta+\omega_{2}<1, \qquad \forall \; t\in
(0, \, \rho).
$$
\end{proposition}
\begin{proof}
First of all, note that the assumption  $\bf h1$ implies, after simple calculation, that
$$\lim_{t \to 0}\frac{tf'(t)-f(t)}{t[1-\beta(f'(t)+1)]} =\lim_{t \to 0}\frac{f'(t)-(f(t)-f(0))/t}{1-\beta(f'(t)+1)}=0. $$
Again, using $\bf h1$, some algebraic manipulation and that $f'$ is convex, we have by Proposition~\ref{pr:conv.aux2}

$$\lim_{t \to 0}\frac{f'(t)+1}{t[1-\beta(f'(t)+1)]}=\lim_{t \to 0}\frac{(f'(t)-f'(0))/t}{1-\beta(f'(t)+1)}=D^+f'(0).
$$
Hence, by combining the two above equalities it is easy to conclude that
$$
\lim_{t \to 0}(1+\vartheta)\omega_{1}\beta\frac{tf'(t)-f(t)+\sqrt{2}c
 \beta[f'(t)+1]}{t[1-\beta(f'(t)+1)]}+\omega_{1}\vartheta+\omega_{2}=(1+\vartheta)\omega_{1}\sqrt{2}c \beta^2D^+f'(0)+\omega_{1}\vartheta+\omega_{2}.
$$
As, $\alpha=\sqrt{2}c \beta^2D^+f'(0)$ and $\omega_{1}(\alpha+\alpha\vartheta+\vartheta)+\omega_{2}<1$, we obtain
that there exists a $\delta>0$ such that
$$
(1+\vartheta)\omega_{1}\beta\frac{tf'(t)-f(t)+\sqrt{2}c
 \beta[f'(t)+1]}{t[1-\beta(f'(t)+1)]}+\omega_{1}\vartheta+\omega_{2}<1, \qquad
 t\in
(0, \delta),
$$
Hence, $\delta\leq \rho$, which proves the first
statement.
To conclude the proof, we use the definition of $\rho$,
the above inequality, and the last part of Proposition~\ref{pr:incr101}.
\end{proof}

\subsection{Relationship of the majorant function with the non-linear function} \label{sec:MFNLO}
In this section we will present the main relationships between
the majorant function $f$ and the function $F$ associated with the nonlinear least square problem.
\begin{lemma} \label{wdns}
Let $x \in \Omega$. If \,$\| x-x_*\|<\min\{\nu,\kappa\}$, then
$F'(x)^* F'(x) $ is invertible and the following
inequalities hold
$$
\left\|F'(x)^{\dagger}\right\|\leq \frac{\beta}{1-\beta [f'(\|
x-x_*\|)+1]}, \qquad  \left\|F'(x)^{\dagger}-F'(x_*)^{\dagger}\right\|
<
\frac{\sqrt{2}\beta^2[f'(\|x-x_*\|)+1]}{1-\beta[f'(\|x-x_*\|)+1]}.
 $$
In particular, $F'(x)^* F'(x)$ is invertible in $B(x_*, r)$.
\end{lemma}
\begin{proof}
Let $x \in \Omega$ such that \,$\| x-x_*\|<\min\{\nu,\kappa\}$. Since $\| x-x_*\|<\nu$, using the definition of $\beta$, the inequality \eqref{Hyp:MH} and last part of Proposition~\ref{pr:incr1} we have
$$
\|F'(x)-F'(x_*)\|\|F'(x_*)^{\dagger}\|\leq
\beta[f'(\| x-x_*\|)-f'(0)] < 1.
$$
For the sake of simplicity, the notations define the following matrices
\begin{equation} \label{eq:daux}
A=F'(x_*), \qquad  B=F'(x), \qquad E=F'(x)-F'(x_*).
\end{equation}
The last definitions, together with the latter inequality, imply that
$$
\|EA^\dagger\|\leq\|E\|\|A^\dagger\| < 1,
$$
which, using that $F'(x_*)$  is injective, implies in view of Lemma \ref{lem:ban2}
that  $F'(x)$ is injective. So, $F'(x)^* F'(x) $ is invertible and by definition of
 $r$ we obtain that $F'(x)^* F'(x)$ is invertible for all $x\in B(x_*, r)$.

We already know that $ F'(x_*)$ and $F'(x)$ are injective. Hence, to conclude the lemma
 use definitions in \eqref{eq:daux} and then combine the above inequality and Lemma \ref{lem:ban}.
\end{proof}

Now, it is convenient to study the linearization error of $F$ at point in~$\Omega$, for which we define
\begin{equation}\label{eq:def.er}
  E_F(x,y):= F(y)-\left[ F(x)+F'(x)(y-x)\right],\qquad y,\, x\in \Omega.
\end{equation}
We will bound this error by the error in the linearization on the
majorant function $f$
\begin{equation}\label{eq:def.erf}
        e_f(t,u):= f(u)-\left[ f(t)+f'(t)(u-t)\right],\qquad t,\,u \in [0,R).
\end{equation}
\begin{lemma}  \label{pr:taylor}
If  $\|x-x_*\|< \kappa$, then there holds $ \|E_F(x, x_*)\|\leq
e_f(\|x-x_*\|, 0). $
\end{lemma}
\begin{proof}
 Since $B(x_*, \kappa)$ is convex,  we obtain that $x_*+\tau(x-x_*)\in B(x_*, \kappa)$, for $0\leq \tau \leq 1$.
 Thus, as $F$ is continuously differentiable in $\Omega$, definition of $E_F$ and some simple manipulations yield
$$
\|E_F(x,x_*)\|\leq \int_0 ^1 \left \|
[F'(x)-F'(x_*+\tau(x-x_*))]\right\|\,\left\|x_*-x\right\| \;
d\tau.
$$
>From  the last inequality and the assumption \eqref{Hyp:MH}, we
obtain
$$
\|E_F(x,x_*)\| \leq \int_0 ^1
\left[f'\left(\left\|x-x_*\right\|\right)-f'\left(\tau\|x-x_*\|\right)\right]\|x-x_*\|\;d\tau.
$$
Evaluating the above integral and using definition of $e_f$, the
statement follows.
\end{proof}
Define the Gauss-Newton step to the functions $F$ by the following equality:
\begin{equation} \label{eq:ns}
S_{F}(x):=-F'(x)^{\dagger}F(x).
\end{equation}
\begin{lemma}  \label{passonewton}
If  $\|x-x_*\|< \min\{\nu, \kappa\}$, then $$\|S_{F}(x)\|\leq \frac{\beta e_f(\|x-x_*\|, 0)+\sqrt{2}c\beta^2[f'(\|x-x_*\|)+1]}{1-\beta[f'(\|x-x_*\|)+1]}+\|x-x_{*}\|.$$
\end{lemma}
\begin{proof}
Using \eqref{eq:ns}, $F'(x_{*})^{*}F(x_*)=0$ and some algebraic manipulation, it follows from \eqref{eq:def.er} that
\begin{align*}
\|S_{F}(x)\|&=\|F'(x)^{\dagger}\left(F(x_{*})-[F(x)+F'(x)(x_{*}-x)]\right)-(F'(x)^{\dagger}-F'(x_*)^{\dagger})F(x_{*})+(x_{*}-x)\|\\
&\leq \|F'(x)^{\dagger}\|\|E_F(x,x_*)\|+\|F'(x)^{\dagger}-F'(x_*)^{\dagger}\|\|F(x_{*})\|+\|x-x_{*}\|.
\end{align*}
So, the last inequality together with the Lemma~\ref{wdns}, Lemma \ref{pr:taylor} and definition of $c,$ imply that
$$\|S_{F}(x)\|\leq
\frac{\beta e_f(\|x-x_*\|, 0)}{1-\beta[f'(\|x-x_*\|)+1]}+\frac{\sqrt{2}c\beta^2[f'(\|x-x_*\|)+1]}{1-\beta[f'(\|x-x_*\|)+1]}+\|x-x_{*}\|,$$
which is equivalent to the desired inequality.
 \end{proof}
\begin{lemma} \label{l:wdef}
Let $\Omega\subseteq \banacha$ be an open set and $F:{\Omega}\to \banachb$ a continuously differentiable function. Let  $x_*\in \Omega,$   $R>0$ and $c,$  $\beta, $ $\kappa$ as a definition in Theorem~\ref{th:nt}. Suppose that $F'(x_*)^*F(x_*)=0$, $F '(x_*)$ is injective and there exists a $f:[0,\; R)\to \mathbb{R}$ continuously differentiable satisfying \eqref{Hyp:MH}, {\bf h1,}  {\bf h2} and {\bf h3}.  Let $\alpha$, $\vartheta$,  $\omega_{1}$, $\omega_{2}$,  $\nu$, $\rho$ and $r$ as in Theorem~\ref{th:nt}. Assume that $x\in B(x_*, r)\backslash \{x_*\}$, i.e., $0<\|x-x_*\|< r$.  Define
\begin{equation} \label{eq:DNSqnG}
x_{+}={x}+S, \qquad  B(x)S=-F'(x)^*F(x)+r,
\end{equation}
where $B(x)$ is an invertible approximation of $F'(x)^*F'(x)$ satisfying
\begin{equation}\label{con:qnG}
\|B(x)^{-1}F'(x)^*F'(x)\| \leq \omega_{1}, \qquad
\|B(x)^{-1}F'(x)^*F'(x)-I\| \leq \omega_{2},
\end{equation}
and the forcing term $\theta$ and the residual $r$ satisfy
\begin{equation}\label{eq:ERROqnG}
\theta \mbox{cond}(P F'(x)^*F'(x))\leq\vartheta, \qquad  \|P r\|\leq \theta\|P F'(x)^*F(x)\|,
\end{equation}
with $P$  an invertible matrix(preconditioner for the linear system in \eqref{eq:DNSqnG}). Then $x_{+}$ is well defined and there holds
  \begin{multline}\label{tt1}
    \|x_{+}-x_*\| \leq
    (1+\vartheta)\omega_{1}\beta\frac{[f'(\|x-x_*\|)\|x-x_*\|-f(\|x-x_*\|)]}{\|x-x_*\|^2[1-\beta(f'(\|x-x_*\|)+1)]}{\|x-x_*\|}^{2}\\+
    \left(\frac{(1+\vartheta)\omega_1\sqrt{2}c\beta^2[f'(\|x-x_*\|)+1]}{\|x-x_*\|[1-\beta(f'(\|x-x_*\|)+1)]}+\omega_1\vartheta+\omega_2\right)\|x-x_*\|,\qquad
    k=0,1,\ldots.
\end{multline}
In particular,
$$
\|x_{+}-x_{*}\|< \|x-x_*\|.
$$
\end{lemma}
\begin{proof}
First note that, as $\|x-x_*\|<r$, it follows from Lemma \ref{wdns} that $F'(x)^*F'(x)$ is invertible. Now, let $B(x)$ an invertible approximation of it satisfying \eqref{con:qnG}. Thus, $x_{+}$ is well defined. Now, as $F'(x_*)^*F(x_*)=0,$ some simple algebraic manipulation and \eqref{eq:DNSqnG} yield
\begin{multline*}
x_{+}-x_{*}=  x -x_* - B(x)^{-1}F'(x)^*\big(F(x)-F(x_*)\big)+B(x)^{-1}{r}\\
+B(x)^{-1}F'(x)^*F'(x)\left[F'(x_*)^{\dagger} F(x_*)- F'(x)^{\dagger}F(x_*)\right].
\end{multline*}
Again, some algebraic manipulation in the above equation gives
\begin{multline*}
x_{+}-x_{*}=B(x)^{-1}F'(x)^*F'(x) F'(x)^\dagger\big(F(x_*)-[F(x)+F'(x)(x_{*}-x)]\big)+B(x)^{-1}{r}\\
+B(x)^{-1}\left(F'(x)^*F'(x)-B(x)\right)(x-x_*) +B(x)^{-1}F'(x)^*F'(x)[F'(x_*)^\dagger F(x_*)- F'(x)^\dagger F(x_*)].
\end{multline*}
The last equation, together with  \eqref{eq:def.er} and \eqref{con:qnG}, imply that
\begin{multline*}
\|x_{+}-x_{*}\|\leq \omega_1 \|F'(x)^\dagger\| \|E_{F}(x,x_{*})\|+\|B(x)^{-1}{r}\|
+\omega_2\|x-x_*\|+\omega_1\|F'(x)^\dagger-F'(x_*)^\dagger\|\|F(x_*)\|.
\end{multline*}
On the other hand, using \eqref{eq:ns}, \eqref{con:qnG} and \eqref{eq:ERROqnG} we have, by simple calculus,
\begin{align*}
\|B(x)^{-1}{r}\|&\leq\|B(x)^{-1}P^{-1}\|\|P{r}\|\\
&\leq\theta\|B(x)^{-1}F'(x)^*F'(x)\|\|(PF'(x)^*F'(x))^{-1}\|\|PF'(x)^*F'(x)\|\|F'(x)^\dagger F(x)\|\\
&\leq\omega_1\vartheta\|S_F(x)\|.
\end{align*}
Hence, it follows from the two last equations that
$$
\|x_{+}-x_{*}\|\leq \omega_1 \|F'(x)^\dagger\| \|E_{F}(x,x_{*})\|
+\omega_1\vartheta\|S_F(x)\|+\omega_2\|x-x_*\|+\omega_1\|F'(x)^\dagger-F'(x_*)^\dagger\|\|F(x_*)\|.
$$
Combining the last equation with the Lemmas~\ref{wdns}, \ref{pr:taylor} and \ref{passonewton}, we obtain that
$$
\|x_{+}-x_{*}\|\leq (1+\vartheta)\beta\omega_1\frac{e_f(\|x-x_*\|,
0)+\sqrt{2}c
\beta(f'(\|x-x_*\|)+1)}{1-\beta(f'(\|x-x_*\|)+1)}+\omega_1\vartheta\|x-x_*\|+\omega_2\|x-x_*\|.
$$
Now, using \eqref{eq:def.erf} and some algebraic manipulation, we conclude from the last inequality that
\begin{multline*}
\|x_{+}-x_{*}\|\leq  (1+\vartheta)\beta\omega_1\frac{f'(\|x-x_*\|)\|x-x_*\|-f(\|x-x_*\|)+\sqrt{2}c
\beta(f'(\|x-x_*\|)+1)}{1-\beta(f'(\|x-x_*\|)+1)}\\+\omega_1\vartheta\|x-x_*\|+\omega_2\|x-x_*\|,
\end{multline*}
which is equivalent to \eqref{tt1}.
To end the proof, note that the right hand side of \eqref{tt1} is equivalent to
$$\Bigg[(1+\vartheta)\omega_1\beta\frac{f'(\|x-x_*\|)\|x-x_*\|-f(\|x-x_*\|)+\sqrt{2}c
\beta(f'(\|x-x_*\|)+1)}{\|x-x_*\|[1-\beta(f'(\|x-x_*\|)+1)]}+\omega_1\vartheta+\omega_2\Bigg]\|x-x_*\|.$$
On the other hand, as $x\in B(x_*,r)/\{x_*\}$, i.e., $0<\|x-x_*\|<r\leq \rho$ we apply the Proposition~\ref{pr:incr102} with $t=\|x-x_*\|$ to conclude that the quantity in the bracket above is less than one. So, the last inequality of the lemma follows.
\end{proof}
\subsection{Proof of {\bf Theorem \ref{th:nt}}} \label{sec:proof}
Now, we will produce the proof of Theorem  \ref{th:nt}.
\begin{proof}
Since  $x_0\in B(x_*,r)/\{x_*\},$ i.e., $0<\|x_0-x_*\|<r,$ by combination of
Lemma~\ref{wdns}, last inequality in Lemma~\ref{l:wdef} and induction argument, it is easy to see that  $\{x_k\}$ is well defined and remains in $B(x_*,r)$.

We are going to prove that  $\{x_k\}$ converges towards $x_*$.
As, $\{x_k\}$ is well defined and contained in  $B(x_*,r)$,
applying Lemma~\ref{l:wdef} with $x_+=x_{k+1},$ $x=x_{k},$ $r=r_{k},$ $B(x)=B(x_{k}),$ $P=P_k,$ and $\theta=\theta_{k}$ we obtain
\begin{multline*}
    \|x_{k+1}-x_*\| \leq
    (1+\vartheta)\omega_{1}\beta\frac{[f'(\|x_k-x_*\|)\|x_k-x_*\|-f(\|x_k-x_*\|)]}{\|x_k-x_*\|^2[1-\beta(f'(\|x_k-x_*\|)+1)]}{\|x_k-x_*\|}^{2}\\+
    \left(\frac{(1+\vartheta)\omega_1\sqrt{2}c\beta^2[f'(\|x_k-x_*\|)+1]}{\|x_k-x_*\|[1-\beta(f'(\|x_k-x_*\|)+1)]}+\omega_1\vartheta+\omega_2\right)\|x_k-x_*\|,\qquad
    k=0,1,\ldots.
 \end{multline*}
Now, using the last inequality of Lemma~\ref{l:wdef}, it is easy to conclude that
\begin{equation}\label{tt}
\|x_{k}-x_*\|< \|x_0-x_*\|, \qquad \;k=1, 2 \ldots.
 \end{equation}
Hence, combining the last two inequalities with the last part of
Proposition~\ref{pr:incr101} we obtain that
\begin{multline*}
    \|x_{k+1}-x_*\| \leq
    (1+\vartheta)\omega_{1}\beta\frac{[f'(\|x_0-x_*\|)\|x_0-x_*\|-f(\|x_0-x_*\|)]}{\|x_0-x_*\|^2[1-\beta(f'(\|x_0-x_*\|)+1)]}{\|x_k-x_*\|}^{2}\\+
    \left(\frac{(1+\vartheta)\omega_1\sqrt{2}c\beta^2[f'(\|x_0-x_*\|)+1]}{\|x_0-x_*\|[1-\beta(f'(\|x_0-x_*\|)+1)]}+\omega_1\vartheta+\omega_2\right)\|x_k-x_*\|,\qquad
    k=0,1,\ldots,
 \end{multline*}
 which is the inequality \eqref{eq:q2}. Now, using \eqref{tt} and the last inequality we have
 \begin{multline*}
    \|x_{k+1}-x_*\| \leq \\
    \Bigg[(1+\vartheta)\omega_1\beta\frac{f'(\|x_0-x_*\|)\|x_0-x_*\|-f(\|x_0-x_*\|)+\sqrt{2}c
\beta(f'(\|x_0-x_*\|)+1)}{\|x_0-x_*\|[1-\beta(f'(\|x_0-x_*\|)+1)]}+\omega_1\vartheta+\omega_2\Bigg]\|x_k-x_*\|,
\end{multline*}
 for all  $k=0,1,\ldots$. Applying Proposition
\ref{pr:incr102} with $t=\|x_0-x_*\|$ it is straightforward to
conclude from the latter inequality that $\{\|x_{k}-x_*\|\}$
converges to zero. So, $\{x_k\}$ converges to $x_*$.
\end{proof}
\section{Special cases} \label{sec:ec}
In this section, we present two special cases of
Theorem~\ref{th:nt}. They include the classical convergence theorem on Gauss-Newton's method under the Lipschitz condition and Smale's theorem on Gauss-Newton for analytical functions.

\subsection{Convergence result for Lipschitz condition}
In this section we show a correspondent theorem for Theorem
\ref{th:nt} under the Lipschitz condition, instead of the general assumption
\eqref{Hyp:MH}.

\begin{theorem}\label{th:ntqnnm}
Let $\Omega\subseteq \banacha$ be an open set,
$F:{\Omega}\to \banachb$ a continuously differentiable
function. Let $x_* \in \Omega,$ $R>0$ and
$$
c:=\|F(x_*)\|, \qquad \beta:=\left\|F'(x_*)^{\dagger}\right\|, \qquad \kappa:=\sup \left\{ t\in [0, R): B(x_*, t)\subset\Omega \right\}.
$$
Suppose that $F'(x_*)^*F(x_*)=0$, $F '(x_*)$ is injective
and there exists a $K>0$ such that
$$
\alpha:=\sqrt{2}c\beta^2 K<1, \qquad \qquad \left\|F'(x)-F'(y)\right\| \leq K\|x-y\|, \qquad \forall\; x,
y\in B(x_*, \kappa).$$
Take $0\leq \vartheta<1$, $0\leq \omega_{2}<\omega_{1}$  such that  $\omega_{1}(\alpha+\alpha\vartheta+\vartheta)+\omega_{2}<1 $. Let
$$r:=\min\left\{\kappa,\frac{2(1-\omega_{1}\vartheta-\omega_{2})-2\sqrt{2}cK\beta^2\omega_{1}(1+\vartheta)}{\beta K\left(2+\omega_{1}-\vartheta\omega_{1}-2\omega_{2}\right)}\right\}.$$
Then, the inexact Gauss-Newton like methods for solving \eqref{eq:p1}, with initial point $x_0\in
B(x_*, r)\backslash \{x_*\}$
\begin{equation} \label{eq:DNSqn223}
x_{k+1}={x_k}+S_k, \qquad  B(x_k)S_k=-F'(x_k)^*F(x_k)+r_{k}, \qquad
\; k=0,1,\ldots,
\end{equation}
with the following conditions for the residual $r_k,$ and the forcing term $\theta_k$
$$
\|P_{k}r_{k}\|\leq \theta_{k}\|P_{k}F'(x_{k})^*F(x_{k})\|,
\qquad  0\leq \theta_{k}\mbox{cond}(P_{k}F'(x_{k})^*F'(x_{k}))\leq
\vartheta,\qquad
\; k=0,1,\ldots,$$
where  $\{P_{k}\}$ is an invertible matrix sequence (preconditoners for the linear system in \eqref{eq:DNSqn223}) and
$B(x_k)$ is an invertible approximation of $F'(x_k)^*F'(x_k)$ satisfying
$$
\|B(x_k)^{-1}F'(x_k)^*F'(x_k)\| \leq \omega_{1}, \qquad
\|B(x_k)^{-1}F'(x_k)^*F'(x_k)-I\| \leq \omega_{2}, \qquad
\; k=0,1,\ldots,
$$
is well defined, contained in $B(x_*,r)$, converges to $x_*$ and there holds
$$   \|x_{k+1}-x_*\| \leq
    \frac{(1+\vartheta)\beta\omega_{1} K}{2(1-\beta K\|x_0-x_*\|)}\|x_k-x_*\|^2+\left(\frac{(1+\vartheta)\omega_{1}\sqrt{2}c \beta^2K}{1-\beta K\|x_0-x_*\|}+\omega_1\vartheta+\omega_2\right)\|x_k-x_*\|,
$$
for all $ k=0,1,\ldots.$
\end{theorem}
 \begin{proof}
It is immediate to prove that  $F$, $x_*$ and $f:[0,
\kappa)\to \mathbb{R}$ as defined by $ f(t)=Kt^{2}/2-t, $ satisfy the
inequality \eqref{Hyp:MH}, conditions {\bf h1} and {\bf h2}. Since $\sqrt{2}c\beta^2 K<1$  the condition {\bf h3} also holds. In this case, it is easy to see that constants
$\nu$ and $\rho$ as defined in Theorem~\ref{th:nt}, satisfy
$$0<\rho=\frac{2(1-\omega_{1}\vartheta-\omega_{2})-2\sqrt{2}cK\beta^2\omega_{1}(1+\vartheta)}{\beta K\left(2+\omega_{1}-\vartheta\omega_{1}-2\omega_{2}\right)}  \leq \nu=1/\beta K ,$$
as a consequence,
$
0<r=\min \{\kappa,\,\rho\}.
$
Therefore, as  $F$, $r$,  $f$ and $x_*$ satisfy all of the
hypotheses of Theorem \ref{th:nt}, taking  $x_0\in B(x_*,
r)\backslash \{x_*\}$ the statements of the theorem follow from
Theorem~\ref{th:nt}.
\end{proof}

For  the case $\vartheta=0$,  the Theorem \ref{th:ntqnnm} becomes:
\begin{corollary}\label{cor:li1}
Let $\Omega\subseteq \banacha$ be an open set,
$F:{\Omega}\to \banachb$ a continuously differentiable
function. Let $x_* \in \Omega,$ $R>0$ and
$$
c:=\|F(x_*)\|, \qquad \beta:=\left\|F'(x_*)^{\dagger}\right\|, \qquad \kappa:=\sup \left\{ t\in [0, R): B(x_*, t)\subset\Omega \right\}.
$$
Suppose that $F'(x_*)^*F(x_*)=0$, $F '(x_*)$ is injective
and there exists a $K>0$ such that
$$
\alpha:=\sqrt{2}c\beta^2 K<1, \qquad \qquad \left\|F'(x)-F'(y)\right\| \leq K\|x-y\|, \qquad \forall\; x,
y\in B(x_*, \kappa).$$
Take  $0\leq \omega_{2}<\omega_{1}$  such that  $\omega_{1}\alpha+\omega_{2}<1 $. Let
$$r:=\min\left\{\kappa,\frac{2(1-\omega_{2})-2\sqrt{2}cK\beta^2\omega_{1}}{\beta K\left(2+\omega_{1}-2\omega_{2}\right)}\right\}.$$
Then, the Gauss-Newton's like method for solving \eqref{eq:p1}, with initial point $x_0\in
B(x_*, r)\backslash \{x_*\}$
\begin{equation} \label{eq:DNSqn2}
x_{k+1}={x_k}+S_k, \qquad  B(x_k)S_k=-F'(x_k)^*F(x_k), \qquad
\; k=0,1,\ldots,
\end{equation}
where $B(x_k)$ is an invertible approximation of $F'(x_k)^*F'(x_k)$ satisfying
$$
\|B(x_k)^{-1}F'(x_k)^*F'(x_k)\| \leq \omega_{1}, \qquad
\|B(x_k)^{-1}F'(x_k)^*F'(x_k)-I\| \leq \omega_{2}, \qquad
\; k=0,1,\ldots,
$$
is well defined, contained in $B(x_*,r)$, converges to $x_*$ and there holds
$$   \|x_{k+1}-x_*\| \leq
    \frac{\beta\omega_{1} K}{2(1-\beta K\|x_0-x_*\|)}\|x_k-x_*\|^2+\left(\frac{\omega_{1}\sqrt{2}c \beta^2K}{1-\beta K\|x_0-x_*\|}+\omega_2\right)\|x_k-x_*\|,
$$
for all $ k=0,1,\ldots.$
\end{corollary}
Note that letting $c=0$ in the above corollary, we obtain the Corollary~6.1 of \cite{chen2}.
\subsection{Convergence result under Smale's condition }
In this section we present a correspondent theorem to Theorem
\ref{th:nt} under Smale's condition. For more details
   see Smale \cite{S86} and Dedieu and Shub  \cite{MR1651750}.

\begin{theorem}\label{theo:Smale}
Let $\Omega\subseteq \banacha$ be an open set,
$F:{\Omega}\to \banachb$ a continuously differentiable
function. Let $x_* \in \Omega,$ $R>0$ and
$$
c:=\|F(x_*)\|, \qquad \beta:=\left\|F'(x_*)^{\dagger}\right\|, \qquad \kappa:=\sup \left\{ t\in [0, R): B(x_*, t)\subset\Omega \right\}.
$$
Suppose that $F'(x_*)^*F(x_*)=0$, $F '(x_*)$ is injective
and
\begin{equation} \label{eq:SmaleCond}
 \qquad \gamma := \sup _{ n > 1 }\left\| \frac
{F^{(n)}(x_*)}{n !}\right\|^{1/(n-1)}<+\infty,\qquad
\qquad \alpha:=2\sqrt{2}c\beta^2\gamma< 1.
\end{equation}
Take $0\leq \vartheta<1$, $0\leq \omega_{2}<\omega_{1}$  such that  $\omega_{1}(\alpha+\alpha\vartheta+\vartheta)+\omega_{2}<1 $. Let
$a:=(1-\vartheta\omega_{1}-\omega_{2})$,
$b:=(1+\vartheta)\omega_{1}\beta,$ $\bar{a}:=b+2a(1+\beta)-\sqrt{2}\gamma\beta bc$ and
$$
r:=\min \left\{\kappa, \frac{\bar{a}-\sqrt{\bar{a}^2-4a(1+\beta)(a-2\sqrt{2}c\beta b \gamma)}}{2a\gamma(1+\beta)}\right\}.
$$
Then, the inexact Gauss-Newton like methods for solving \eqref{eq:p1}, with initial point $x_0\in
B(x_*, r)\backslash \{x_*\}$
\begin{equation} \label{eq:DNSqn22}
x_{k+1}={x_k}+S_k, \qquad  B(x_k)S_k=-F'(x_k)^*F(x_k)+r_{k}, \qquad
\; k=0,1,\ldots,
\end{equation}
with the following conditions for  the residual $r_k,$ and the forcing term $\theta_k$
$$
\|P_{k}r_{k}\|\leq \theta_{k}\|P_{k}F'(x_{k})^*F(x_{k})\|,
\qquad  0\leq \theta_{k}\mbox{cond}(P_{k}F'(x_{k})^*F'(x_{k}))\leq
\vartheta,\qquad
\; k=0,1,\ldots,$$
where  $\{P_{k}\}$ is an invertible matrix sequence (preconditoners for the linear system in \eqref{eq:DNSqn22}) and
$B(x_k)$ is an invertible approximation of $F'(x_k)^*F'(x_k)$ satisfying
$$
\|B(x_k)^{-1}F'(x_k)^*F'(x_k)\| \leq \omega_{1}, \qquad
\|B(x_k)^{-1}F'(x_k)^*F'(x_k)-I\| \leq \omega_{2}, \qquad
\; k=0,1,\ldots,
$$
is well defined, contained in $B(x_*,r)$, converges to $x_*$ and there holds
\begin{align*}
    \|x_{k+1}-x_*\| \leq &
    \frac{(1+\vartheta)\omega_{1}\beta\gamma}{(1-\gamma \|x_0-x_*\|)^2-\beta\gamma(2\|x_0-x_*\|-\gamma\|x_0-x_*\|^2)}\|x_k-x_*\|^2\\+
    &\left(\frac{(1+\vartheta)\omega_{1}\sqrt{2}c \beta^2\gamma(2-\gamma\|x_0-x_*\|)}{(1-\gamma \|x_0-x_*\|)^2-\beta\gamma(2\|x_0-x_*\|-\gamma\|x_0-x_*\|^2)}+\omega_1\vartheta+\omega_2\right)\|x_k-x_*\|,
\end{align*}
for all $ k=0,1,\ldots.$
 \end{theorem}

We need the following result to prove the above theorem.
\begin{lemma} \label{lemma:qc1}
Let $\Omega\subseteq \banacha$ be an open set,
$F:{\Omega}\to \banachb$ an analytic function.
 Suppose that
$x_*\in \Omega$  and  $B(x_{*},
1/\gamma) \subset \Omega$, where $\gamma$ is defined in
\eqref{eq:SmaleCond}. Then, for all $x\in B(x_{*}, 1/\gamma)$
there holds
$$
\|F''(x)\| \leqslant  2\gamma/( 1- \gamma \|x-x_*\|)^3.
$$
\end{lemma}
\begin{proof}
See the proof  of  the Lemma~21 of \cite{MAX2}.
\end{proof}
The next result gives a condition that is easier to check than
condition \eqref{Hyp:MH}, whenever the functions under
consideration are twice continuously differentiable.
\begin{lemma} \label{lc}
Let $\Omega\subseteq \banacha$ be an open set, $x_*\in \Omega$  and
$F:{\Omega}\to \banachb$  be twice continuously on $\Omega$. If there exists a \mbox{$f:[0,R)\to \mathbb {R}$} twice continuously differentiable such that
 \begin{equation} \label{eq:lc2}
\|F''(x)\|\leqslant f''(\|x-x_*\|),
\end{equation}
for all $x\in  \Omega$ such that  $\|x-x_*\|<R$. Then $F$ and $f$
satisfy \eqref{Hyp:MH}.
\end{lemma}
\begin{proof}
See the proof  of  the Lemma~22 of \cite{MAX2}.
\end{proof}
{\bf [Proof of Theorem \ref{theo:Smale}]}.  Consider the real
function $f:[0,1/\gamma) \to \mathbb{R}$ defined by
$$
f(t)=\frac{t}{1-\gamma t}-2t.
$$
It is straightforward to show that $f$ is  analytic and that
$$
f(0)=0, \quad f'(t)=1/(1-\gamma t)^2-2, \quad f'(0)=-1, \quad
f''(t)=(2\gamma)/(1-\gamma t)^3, \quad f^{n}(0)=n!\,\gamma^{n-1},
$$
for $n\geq 2$. It follows from the last equalities
that $f$ satisfies {\bf h1} and  {\bf h2}. Since $ 2\sqrt{2}c\beta^2\gamma< 1$ the condition  {\bf h3} also holds. Now, as
$f''(t)=(2\gamma)/(1-\gamma t)^3$ combining Lemmas ~\ref{lc},
\ref{lemma:qc1}  we conclude that $F$  and $f$ satisfy
\eqref{Hyp:MH} with $R=1/\gamma$. In this case, it is easy to see
that constants $\nu$ and $\rho$ as defined in Theorem \ref{th:nt}, satisfy
$$
0<\rho=\frac{\bar{a}-\sqrt{\bar{a}^2-4a(1+\beta)(a-2\sqrt{2}c\beta b \gamma)}}{2a\gamma(1+\beta)}<\nu=((1+\beta)-\sqrt{\beta(1+\beta)})/(\gamma(1+\beta))<1/\gamma,
$$
and as a consequence,
$
0<r=\min \{\kappa,\rho\}.
$
Therefore, as  $F$,
$\sigma$, $f$ and $x_*$ satisfy all hypotheses of Theorem
\ref{th:nt}, taking $x_0\in B(x_*, r)\backslash \{x_*\}$, the
statements of the theorem follow from Theorem \ref{th:nt}.\qed
\\

For the case $\vartheta=0$,  the Theorem \ref{theo:Smale} becomes:
\begin{corollary}\label{cor:tt}
Let $\Omega\subseteq \banacha$ be an open set,
$F:{\Omega}\to \banachb$ a continuously differentiable
function. Let $x_* \in \Omega,$ $R>0$ and
$$
c:=\|F(x_*)\|, \qquad \beta:=\left\|F'(x_*)^{\dagger}\right\|, \qquad \kappa:=\sup \left\{ t\in [0, R): B(x_*, t)\subset\Omega \right\}.
$$
Suppose that $F'(x_*)^*F(x_*)=0$, $F '(x_*)$ is injective
and
$$
 \qquad \gamma := \sup _{ n > 1 }\left\| \frac
{F^{(n)}(x_*)}{n !}\right\|^{1/(n-1)}<+\infty,\qquad
\qquad \alpha:=2\sqrt{2}c\beta^2\gamma< 1.
$$
Take  $0\leq \omega_{2}<\omega_{1}$  such that  $\omega_{1}\alpha+\omega_{2}<1 $. Let
$\bar{a}:=\omega_{1}\beta+2(1-\omega_{2})(1+\beta)-\sqrt{2}\gamma\beta^2\omega_{1}c$ and
$$
r:=\min \left\{\kappa, \frac{\bar{a}-\sqrt{\bar{a}^2-4(1-\omega_{2})(1+\beta)(1-\omega_{2}-2\sqrt{2}c\beta^2\omega_{1} \gamma)}}{2(1-\omega_{2})\gamma(1+\beta)}\right\}.
$$
Then, the Gauss-Newton's like method for solving \eqref{eq:p1}, with initial point $x_0\in
B(x_*, r)\backslash \{x_*\}$
$$
x_{k+1}={x_k}+S_k, \qquad  B(x_k)S_k=-F'(x_k)^*F(x_k), \qquad
\; k=0,1,\ldots,
$$
where $B(x_k)$ is an invertible approximation of $F'(x_k)^*F'(x_k)$ satisfying
$$
\|B(x_k)^{-1}F'(x_k)^*F'(x_k)\| \leq \omega_{1}, \qquad
\|B(x_k)^{-1}F'(x_k)^*F'(x_k)-I\| \leq \omega_{2}, \qquad
\; k=0,1,\ldots,
$$
is well defined, contained in $B(x_*,r)$, converges to $x_*$ and there holds
\begin{align*}
    \|x_{k+1}-x_*\| \leq &
    \frac{\omega_{1}\beta\gamma}{(1-\gamma \|x_0-x_*\|)^2-\beta\gamma(2\|x_0-x_*\|-\gamma\|x_0-x_*\|^2)}\|x_k-x_*\|^2\\+
    &\left(\frac{\omega_{1}\sqrt{2}c \beta^2\gamma(2-\gamma\|x_0-x_*\|)}{(1-\gamma \|x_0-x_*\|)^2-\beta\gamma(2\|x_0-x_*\|-\gamma\|x_0-x_*\|^2)}+\omega_2\right)\|x_k-x_*\|,
\end{align*}
for all $ k=0,1,\ldots.$
 \end{corollary}
Note that letting $c=0$ in the above corollary, we obtain the Example~1 of \cite{chen2}.
\section{Final remark}\label{rf}
The Theorem \ref{th:nt} gives an estimate of the convergence radius for inexact Gauss-Newton like methods. In particular, for $\vartheta=\omega_1=0$ and $\omega_2=1$  is shown in Ferreira et al. \cite{MAX2}, that $r$ is the best possible convergence radius.

Another detail is that, as pointed out by  Morini in  \cite{B10} if preconditioning  $P_{k}$,  satisfying
\begin{equation} \label{eq:r1}
\|P_{k}r_{k}\|\leq \theta_{k}\|P_{k}F'(x_{k})^*F(x_k)\|,
\end{equation}
for some forcing sequence $\{\theta_{k}\}$, is applied to finding the inexact Gauss-Newton steep, then the inverse proportionality between each forcing term $\theta_{k}$ and $\mbox{cond}(P_{k}F'(x_{k})^*F(x_k))$ stated in the following assumption:
\begin{equation} \label{eq:r2}
0< \theta_{k}\mbox{cond}(P_{k}F'(x_{k})^*F(x_k))\leq
\vartheta,\qquad
\; k=0,1,\ldots,
\end{equation}
is sufficient to guarantee convergence, and may be overly restrictive to bound the sequence $\{\theta_{k}\}$, always such that the matrices $P_{k}F'(x_{k})^*F(x_k)$, for $k=0,1,\ldots,$ are badly conditioned. Moreover,   $\theta_{k}$ does not depend on $\mbox{cond} (F'(x_{k})^*F(x_k))$ but only on the $\mbox{cond}(P_{k}F'(x_{k})^*F(x_k))$ and a suitable choice of scaling matrix $P_k$ leads to a relaxation of the forcing terms.


\begin{thebibliography}{10}

\bibitem{MR1918655}
R.~L. Adler, J.~P. Dedieu, J.~Y. Margulies, M.~Martens, and M.~Shub.
\newblock {N}ewton's method on {R}iemannian manifolds and a geometric model for
  the human spine.
\newblock {\em IMA J. Numer. Anal.}, 22(3):359--390, 2002.

\bibitem{1390141}
F.~Alvarez, J.~Bolte, and J.~Munier.
\newblock A unifying local convergence result for {N}ewton's method in
  {R}iemannian manifolds.
\newblock {\em Found. Comput. Math.}, 8(2):197--226, 2008.

\bibitem{Argyros2010}
I.~K. Argyros and S.~Hilout.
\newblock Improved generalized differentiability conditions for {N}ewton-like
  methods.
\newblock {\em J. Complexity}, In Press, Corrected Proof:--, 2010.

\bibitem{BCSS97}
L.~Blum, F.~Cucker, M.~Shub, and S.~Smale.
\newblock {\em Complexity and real computation}.
\newblock Springer-Verlag, New York, 1998.
\newblock With a foreword by Richard M. Karp.

\bibitem{AAA}
J.~Chen.
\newblock The convergence analysis of inexact {G}auss-{N}ewton methods for
  nonlinear problems.
\newblock {\em Comput. Optim. Appl.}, 40(1):97--118, 2008.

\bibitem{C11}
J.~Chen and W.~Li.
\newblock Convergence of {G}auss-{N}ewton's method and uniqueness of the
  solution.
\newblock {\em Appl. Math. Comput.}, 170(1):686--705, 2005.

\bibitem{C10}
J.~Chen and W.~Li.
\newblock Convergence behaviour of inexact {N}ewton methods under weak
  {L}ipschitz condition.
\newblock {\em J. Comput. Appl. Math.}, 191(1):143--164, 2006.

\bibitem{chen2}
J.~Chen and W.~Li.
\newblock Local convergence results of {G}auss-{N}ewton's like method in weak
  conditions.
\newblock {\em J. Math. Anal. Appl.}, 324(2):1381 -- 1394, 2006.

\bibitem{MR1895083}
J.~P. Dedieu and M.~H. Kim.
\newblock {N}ewton's method for analytic systems of equations with constant
  rank derivatives.
\newblock {\em J. Complexity}, 18(1):187--209, 2002.

\bibitem{Jean-PierreDedieu07012003}
J.~P. Dedieu, P.~Priouret, and G.~Malajovich.
\newblock {N}ewton's method on {R}iemannian manifolds: covariant alpha theory.
\newblock {\em IMA J. Numer. Anal.}, 23(3):395--419, 2003.

\bibitem{MR1651750}
J.~P. Dedieu and M.~Shub.
\newblock {N}ewton's method for overdetermined systems of equations.
\newblock {\em Math. Comp.}, 69(231):1099--1115, 2000.

\bibitem{DE1}
R.~S. Dembo, S.~C. Eisenstat, and T.~Steihaug.
\newblock Inexact {N}ewton methods.
\newblock {\em SIAM J. Numer. Anal.}, 19(2):400--408, 1982.

\bibitem{DN1}
J.~E. Dennis, Jr. and R.~B. Schnabel.
\newblock {\em Numerical methods for unconstrained optimization and nonlinear
  equations}, volume~16 of {\em Classics in Applied Mathematics}.
\newblock Society for Industrial and Applied Mathematics (SIAM), Philadelphia,
  PA, 1996.
\newblock Corrected reprint of the 1983 original.

\bibitem{F08}
O.~P. Ferreira.
\newblock Local convergence of {N}ewton's method in {B}anach space from the
  viewpoint of the majorant principle.
\newblock {\em IMA J. Numer. Anal.}, 29(3):746--759, 2009.

\bibitem{MAX1}
O.~P. Ferreira and M.~L.~N. Goncalves.
\newblock Local convergence analysis of inexact {N}ewton-like methods under
  majorant condition.
\newblock {\em Comput. Optim. Appl.}, Article in Press:1--21, 2009.

\bibitem{MAX2}
O.~P. Ferreira, M.~L.~N. Goncalves, and P.~R. Oliveira.
\newblock Local convergence analysis of gauss-newton's method under majorant
  condition. To appear in  \newblock {\em J. Complexity}, 1--20, 2010.

\bibitem{FS06}
O.~P. Ferreira and B.~F. Svaiter.
\newblock Kantorovich's majorants principle for {N}ewton's method.
\newblock {\em Comput. Optim. Appl.}, 42(2):213--229, 2009.

\bibitem{HL93}
J.-B. Hiriart-Urruty and C.~Lemar{\'e}chal.
\newblock {\em Convex analysis and minimization algorithms. {I}}, volume 305 of
  {\em Grundlehren der Mathematischen Wissenschaften [Fundamental Principles of
  Mathematical Sciences]}.
\newblock Springer-Verlag, Berlin, 1993.
\newblock Fundamentals.

\bibitem{KAN1}
L.~V. Kantorovich.
\newblock The principle of the majorant and {N}ewton's method.
\newblock {\em Doklady Akad. Nauk SSSR (N.S.)}, 76:17--20, 1951.

\bibitem{Li2010}
C.~Li, N.~Hu, and J.~Wang.
\newblock Convergence behavior of {G}auss-{N}ewton's method and extensions of
  the smale point estimate theory.
\newblock {\em J. Complexity}, In Press, Corrected Proof:--, 2010.

\bibitem{LZJ}
C.~Li, W.~H. Zhang, and X.~Q. Jin.
\newblock Convergence and uniqueness properties of {G}auss-{N}ewton's method.
\newblock {\em Comput. Math. Appl.}, 47(6-7):1057--1067, 2004.

\bibitem{JM10}
J.~M. Martinez and L.~Qi.
\newblock Inexact {N}ewton methods for solving nonsmooth equations.
\newblock {\em J. Comput. Appl. Math.}, 60:127--145, 1999.

\bibitem{B10}
B.~Morini.
\newblock Convergence behaviour of inexact {N}ewton methods.
\newblock {\em Math. Comp.}, 68:1605--1613, 1999.

\bibitem{MR2475307}
P.~D. Proinov.
\newblock General local convergence theory for a class of iterative processes
  and its applications to {N}ewton's method.
\newblock {\em J. Complexity}, 25(1):38 -- 62, 2009.

\bibitem{Proinov20103}
P.~D. Proinov.
\newblock New general convergence theory for iterative processes and its
  applications to {N}ewton-{K}antorovich type theorems.
\newblock {\em J. Complexity}, 26(1):3 -- 42, 2010.

\bibitem{S86}
S.~Smale.
\newblock {N}ewton's method estimates from data at one point.
\newblock In {\em The merging of disciplines: new directions in pure, applied,
  and computational mathematics ({L}aramie, {W}yo., 1985)}, pages 185--196.
  Springer, New York, 1986.

\bibitem{G1}
G.~W. Stewart.
\newblock On the continuity of the generalized inverse.
\newblock {\em SIAM J. Appl. Math.}, 17:33--45, 1969.

\bibitem{XW10}
X.~Wang.
\newblock Convergence of {N}ewton's method and uniqueness of the solution of
  equations in {B}anach space.
\newblock {\em IMA J. Numer. Anal.}, 20(1):123--134, 2000.

\bibitem{W1}
P.~A. Wedin.
\newblock Perturbation theory for pseudo-inverses.
\newblock {\em Nordisk Tidskr. Informationsbehandling (BIT)}, 13:217--232,
  1973.

\end{thebibliography}

\def\cprime{$'$}

\end{document}